\documentclass{article}
\usepackage{amsmath, amssymb, amsfonts, amsthm, bm, mathtools, dsfont}
\usepackage{tikz}
\usepackage{mathrsfs}

\let\originalleft\left
\let\originalright\right
\renewcommand{\left}{\mathopen{}\mathclose\bgroup\originalleft}
\renewcommand{\right}{\aftergroup\egroup\originalright}

\newcommand{\abs}[1]{\left\vert #1 \right\vert}
\newcommand{\diag}[1]{\mathrm{diag} \left( #1 \right)}
\newcommand{\ind}[1]{\mathds{1}_{\left\lbrace #1 \right\rbrace}}

\newcommand{\norm}[1]{\left\| #1 \right\|}
\newcommand{\transpose}{\mathsf{T}}

\newcommand{\dd}{\mathrm{d}}

\newcommand{\expect}{\mathbb{E}}
\newcommand{\expectation}[1]{\expect \left( #1 \right)}

\newcommand{\genQ}{\mathcal{Q}}

\newcommand{\prequadvar}[1]{\left\langle #1 \right\rangle}
\newcommand{\stdot}{\mathbin{\bm{\cdot}}}

\newcommand{\eps}{\epsilon}

\theoremstyle{plain}

\newtheorem{theorem}{Theorem}[section]

\newtheorem{lemma}[theorem]{Lemma}

\theoremstyle{definition}

\title{Weak convergence of stochastic integrals with respect to the state occupation measure of a Markov chain}
\author{H.\ M.\ Jansen}

\begin{document}
\maketitle

\begin{abstract}
\noindent Our aim is to find sufficient conditions for weak convergence of stochastic integrals with respect to the state occupation measure of a Markov chain. First, we study properties of the state indicator function and the state occupation measure of a Markov chain. In particular, we establish weak convergence of the state occupation measure under a scaling of the generator matrix. Then, relying on the connection between the state occupation measure and the Dynkin martingale related to the state indicator function, we provide sufficient conditions for weak convergence of stochastic integrals with respect to the state occupation measure.
\end{abstract}

\newline

\noindent $^{1}$ Korteweg-de Vries Institute for Mathematics, University of Amsterdam, Science Park 904, 1098 XH Amsterdam, the Netherlands.
\newline

\noindent $^{2}$ TELIN, Ghent University, Sint-Pietersnieuwstraat 41, B-9000 Ghent, Belgium.
\newline

\noindent {\it E-mail}. {\tt h.m.jansen@uva.nl}

\section{Introduction}
We are interested in weak convergence of a specific class of stochastic integrals which arises in the analysis of Markov-modulated queueing systems. More specifically, we would like to find conditions under which
\begin{align}
H_{n}^{-} \stdot G_{n} \Rightarrow H^{-} \stdot G.
\label{eq:introconv}
\end{align}
Here, $X \stdot Y$ denotes the It\^{o} integral of $X$ with respect to $Y$ and $\Rightarrow$ denotes weak convergence. In addition, $H_{n}$ and $G_{n}$ are stochastic processes satisfying $H_{n} \Rightarrow H$ and $G_{n} \Rightarrow G$, with $H_{n}$ being a suitable integrand and $G_{n}$ denoting the (scaled and centered) state occupation measure of an irreducible continuous-time Markov chain.

Rather remarkably, this case does not seem to be covered by the known results dealing with convergence as in Eq.\ \eqref{eq:introconv}. Indeed, to guarantee convergence as in Eq.\ \eqref{eq:introconv}, it is typically required that $G_{n}$ is a martingale or that $G_{n}$ satisfies the P-UT condition. Neither of these requirements is satisfied when $G_{n}$ is the state occupation measure of a Markov chain, even though $G_{n}$ has very nice convergence properties in this case.
Nevertheless, we can find conditions under which the weak convergence in Eq.\ \eqref{eq:introconv} does hold. The key insight underlying these conditions is that we should put restrictions on the total variation of $H_{n}$.

The remainder of this note is organized as follows. In Section \ref{sec:preliminaries}, we derive basic properties of an irreducible, continuous-time Markov chain, its Dynkin martingale, and its state occupation measure. In Section \ref{sec:markovintegral}, we state and prove our main result, which gives conditions that guarantee weak convergence of stochastic integrals with respect to the state occupation measure of an irreducible Markov chain. In Section \ref{sec:boundedvariation}, we have collected some auxiliary results concerning functions of bounded variation, which are used to prove the convergence of the stochastic integrals.

\section{Preliminaries}
\label{sec:preliminaries}
\subsection{Basic properties of Markov chains}
\label{subsec:basicmarkovchain}
Let $J$ be a continuous-time Markov chain with state space $\left\lbrace 1 , \dotsc , d \right\rbrace$ for some $d \in \mathbb{N}$. Let $\genQ$ denote the $d \times d$ generator matrix corresponding to $J$. The state indicator function of $J$ is the $\mathbb{R}^{d}$-valued function $K$ defined via
\begin{align*}
K \left( i ; t \right)
&= \ind{J \left( t \right) = i}
\end{align*}
for $i \in \left\lbrace 1 , \dotsc , d \right\rbrace$ and $t \geq 0$. The function $K$ plays an important role via the state occupation measure, which is the vector-valued stochastic process
\begin{align*}
L \left( t \right) = \int_{0}^{t} K \left( s \right) \, \dd s.
\end{align*}
On an intuitive level, the state indicator function $K$ registers in which state $J$ is, while the state occupation measure $L$ measures how much time $J$ has spent in each state up to a certain time.

Anticipating upcoming results, we derive some equalities. Assume that the generator matrix $\genQ$ is irreducible with a $d \times 1$ column vector $\pi$ denoting its stationary distribution, i.e., $\pi$ is the unique probability vector solving the equation $\pi^{\transpose} \genQ = 0$. Additionally, let $D$ denote the deviation matrix corresponding to $\genQ$; its entries are given by
\begin{align*}
D_{ij} = \int_{0}^{\infty} \left( \mathbb{P} \left( J \left( s \right) = j \, \middle\vert \, J \left( 0 \right) = i \right) - \pi_{j} \right) \, \dd s.
\end{align*}
The integral is well defined, because the irreducibility of $\genQ$ implies that the probability $\mathbb{P} \left( J \left( t \right) = j \, \middle\vert \, J \left( 0 \right) = i \right)$ converges exponentially fast to $\pi_{j}$ as $t \to \infty$ (cf.\ \cite[p.\ 356]{CV2002}). Thus, the deviation matrix $D$ provides a measure for how much the Markov chain $J$ deviates from its stationary distribution when it starts in a fixed point.

Following \cite{CV2002}, we define the ergodic matrix $\Pi = \mathsf{1} \pi^{\transpose}$ and the fundamental matrix $F = D + \Pi$, where $\mathsf{1}$ denotes a $d \times 1$ vector with each entry being $1$. 
Some straightforward arguments (cf.\ \cite{CV2002}) demonstrate that
\begin{align}
\genQ F = F \genQ = \Pi - I = D \genQ = \genQ D
\label{eq:FgenQ}
\end{align}
and
\begin{align}
\pi^{\transpose} D = 0.
\label{eq:pitransposeD}
\end{align}
Applying these identities, we find that
\begin{align*}
\left( \genQ F \right)^{\transpose} \diag{\pi} F
&= \left( \genQ D \right)^{\transpose} \diag{\pi} D + \left( \genQ D \right)^{\transpose} \diag{\pi} \Pi
\end{align*}
and
\begin{align*}
\left( \genQ D \right)^{\transpose} \diag{\pi} D
&= \left( \mathsf{1} \pi^{\transpose} - I \right)^{\transpose} \diag{\pi} D\\
&= \pi \mathsf{1}^{\transpose} \diag{\pi} D - \diag{\pi} D\\
&= \pi \pi^{\transpose} D - \diag{\pi} D\\
&= - \diag{\pi} D.
\end{align*}
Moreover, it holds that
\begin{align*}
\left( \genQ D \right)^{\transpose} \diag{\pi} \Pi
= \left( \genQ D \right)^{\transpose} \pi \pi^{\transpose} 
= \left( \genQ D \right)^{\transpose} \left( \pi \pi^{\transpose} \right)^{\transpose} 
= \left( \pi \pi^{\transpose} \genQ D \right)^{\transpose} 
= 0,
\end{align*}
so
\begin{align}
F^{\transpose} \left( \genQ^{\transpose} \diag{\pi} + \diag{\pi} \genQ \right) F = - \left( \diag{\pi} D + D^{\transpose} \diag{\pi} \right). \label{eq:quadvardevmat}
\end{align}
Given an irreducible generator matrix $\genQ$, the vectors and matrices $\mathsf{1}$, $\pi$, $\Pi$, $F$, and $D$ will be as described above, unless stated otherwise.

\subsection{The Dynkin martingale of a Markov chain}
\label{subsec:markovdynkin}
Markov chains are closely connected to martingales via Dynkin's formula. We will rely heavily on this when proving weak convergence of state occupation measures via the Martingale Central Limit Theorem (MCLT). 

In the next result, we define a martingale $Y$, which is the Dynkin martingale. Additionally, we note that $Y$ is a locally square-integrable martingale. For this class of martingales there are some very strong convergence results available, which typically depend on the predictable quadratic variation process (also called the compensator) of such martingales converging in a suitable manner. We would like to invoke those convergence results later on, so we present the explicit form of the compensator of $Y$ as well. 

\begin{theorem}
\label{thm:dynkinmartingale}
Let $d \in \mathbb{N}$ and let $\genQ$ be a $d \times d$ generator matrix. Let $J$ be a continuous-time Markov chain with state space $\left\lbrace 1 , \dotsc , d \right\rbrace$, generator matrix $\genQ$, and state indicator function $K$. Then the process $Y$ defined via
\begin{align}
Y \left( t \right) 
&= K \left( t \right) - K \left( 0 \right) - \int_{0}^{t} \genQ^{\transpose} K \left( s \right) \, \dd s
\label{eq:dynkin}
\end{align}
is a c\`{a}dl\`{a}g martingale having predictable quadratic variation process
\begin{align}
\prequadvar{Y} \left( t \right)
&= \int_{0}^{t} \diag{\genQ^{\transpose} K \left( s \right)} \, \dd s - \int_{0}^{t} \genQ^{\transpose} \diag{K \left( s \right)} \, \dd s - \int_{0}^{t} \diag{K \left( s \right)} \genQ \, \dd s \label{eq:prequadvardynkin}
\end{align}
and satisfying
\begin{align*}
\expectation{ Y \left( t \right)^{\transpose} Y \left( t \right) } < \infty
\end{align*}
for all $t \geq 0$.
\end{theorem}
\begin{proof}
See \cite[Lem.~2.6.18]{ae2004} and \cite[Lem.~3.8.5]{ae2004}.
\end{proof}
Given a continuous-time Markov chain $J$ as in Theorem \ref{thm:dynkinmartingale}, we will call the process $Y$ defined above the Dynkin martingale associated with $J$.

The last statement of Theorem \ref{thm:dynkinmartingale} implies, in some sense, that $Y$ is a square-integrable martingale. However, the definitions of this term differ throughout the literature: $Y$ is square integrable in the terminology of \cite[Def.\ 1.5.1]{ks1998}, but at this point it is not clear whether $Y$ is square integrable in the terminology of \cite[Def.\ I.1.41]{js2003}. For us, this is not really important, because the theorem implies that $Y$ is locally square integrable for any (reasonable) definition of a square-integrable martingale. This is sufficient for our purposes.

\subsection{Weak convergence of the state occupation measure of a Markov chain}
\label{subsec:stateoccupationmeasureconv}

In the previous subsection, we have defined the Dynkin martingale corresponding to a Markov chain and presented some properties of this martingale. Here, we will leverage these results to obtain the most important results of this section, namely convergence in probability and weak convergence of the state occupation measure of a scaled Markov chain.

The first result is basically the ergodic theorem for irreducible Markov chains. It states that, under a specific scaling, the state occupation measure of an irreducible Markov chain converges uoc in probability to the stationary distribution. On a more intuitive level, this means that a background Markov chain will be close to equilibrium under this specific scaling.

\begin{theorem}
\label{thm:ergodic}
Let $\alpha > 0$ and $d \in \mathbb{N}$. Let $\genQ$ be a $d \times d$ irreducible generator matrix and let $J_{n}$ be a continuous-time Markov chain with state space $\left\lbrace 1 , \dotsc , d \right\rbrace$, generator matrix $n^{\alpha} \genQ$, and state indicator function $K_{n}$. Then, for $n \to \infty$, it holds that
\begin{align*}
\sup_{0 \leq t \leq T} \norm{ n^{\alpha / 2 - \eps} \int_{0}^{t} \left( K_{n} \left( s \right) - \pi \right) \, \dd s}
\end{align*}
converges to $0$ in probability for each $\eps > 0$ and $T > 0$.
\end{theorem}
\begin{proof}
Let $Y_{n}$ be the Dynkin martingale associated with $J_{n}$. We would like to apply the Martingale Central Limit Theorem (MCLT) to derive convergence of $n^{-\alpha / 2 - \eps} Y_{n}$ to the zero process, from which we will get convergence of the state occupation measure.

To be able to apply the MCLT (cf.\ \cite[Th.\ 2.1]{Whitt2007}), we have to verify several properties: we need convergence of the predictable quadratic variation process $\prequadvar{n^{-\alpha / 2 - \eps} Y_{n}}$, together with bounds on the maximum jump sizes of $n^{-\alpha / 2 - \eps} Y_{n}$ and $\prequadvar{n^{-\alpha / 2 - \eps} Y_{n}}$.

We obtain from Theorem \ref{thm:dynkinmartingale} that
\begin{align*}
&\prequadvar{ n^{-\alpha / 2 - \eps} Y_{n} } \left( t \right) \\
&\quad {} = n^{-\alpha - 2 \eps} \prequadvar{Y_{n} } \left( t \right) \\
&\quad {} = n^{-\alpha - 2 \eps} \int_{0}^{t} \diag{n^{\alpha} \genQ^{\transpose} K_{n} \left( s \right)} \, \dd s \\
&\quad \phantom{{} = {}} {} - n^{-\alpha - 2 \eps} \int_{0}^{t} n^{\alpha}\genQ^{\transpose} \diag{K_{n} \left( s \right)} \, \dd s - n^{-\alpha - 2 \eps} \int_{0}^{t} \diag{K_{n} \left( s \right)} n^{\alpha} \genQ \, \dd s.
\end{align*}
Because $K_{n}$ is bounded by $1$, it follows that $\prequadvar{n^{-\alpha / 2 - \eps} Y_{n}}$ converges to the zero process uniformly on compact intervals.

Moreover, $\prequadvar{n^{-\alpha / 2 - \eps} Y_{n}}$ is continuous, and the maximum jump size of each entry of $n^{-\alpha / 2 - \eps} Y_{n}$ is obviously bounded by $n^{-\alpha / 2 - \eps}$. Hence, the maximum jump size of $n^{-\alpha / 2 - \eps} Y_{n}$ and $\prequadvar{n^{-\alpha / 2 - \eps} Y_{n}}$ converges to $0$ as $n \to \infty$.
Then it follows from the MCLT (as presented in \cite[Th.\ 2.1]{Whitt2007}) that $n^{-\alpha / 2 - \eps} Y_{n}$ converges weakly to a Brownian motion whose predictable quadratic variation process is given by the zero process. In other words, $n^{-\alpha / 2 - \eps} Y_{n}$ converges weakly to the zero process.

Recalling that $K_{n}$ is bounded by $1$ and keeping in mind that $n^{-\alpha / 2 - \eps} Y_{n}$ converges weakly to the zero process, it is easy to see from the definition of $Y_{n}$ in Eq.\ \eqref{eq:dynkin} that the process
\begin{align*}
- n^{-\alpha / 2 - \eps} \int_{0}^{t} n^{\alpha} \genQ^{\transpose} K _{n} \left( s \right) \, \dd s
&= - n^{\alpha / 2 - \eps} \int_{0}^{t} \genQ^{\transpose} K _{n} \left( s \right) \, \dd s
\end{align*}
converges weakly to the zero process. Then
\begin{align}
- n^{\alpha / 2 - \eps} \int_{0}^{t} F^{\transpose} \genQ^{\transpose} K _{n} \left( s \right) \, \dd s
\label{eq:Ftransposetimesdynkin}
\end{align}
must converge weakly to the zero process, too. (Although $F$ denotes the fundamental matrix here, it could be any $d \times d$ matrix, of course.)

Now recall the matrix equalities related to the deviation matrix $D$ and the fundamental matrix $F$ that we derived earlier. From these equalities we obtain that
\begin{align}
F^{\transpose} \genQ^{\transpose} K_{n} \left( s \right)
&= \left( \genQ F \right)^{\transpose} K_{n} \left( s \right) = \left( \pi \mathsf{1}^{\transpose} - I \right) K_{n} \left( s \right) = \pi - K_{n} \left( s \right).
\label{eq:FtransposeQtransposeKn}
\end{align}
Combining this with the convergence of the process in Eq.\ \eqref{eq:Ftransposetimesdynkin}, it immediately follows that
\begin{align*}
n^{\alpha / 2 - \eps} \int_{0}^{t} \left( K_{n} \left( s \right) - \pi \right) \, \dd s
\end{align*}
converges weakly to the zero process. Because the zero process is a deterministic limit, the convergence actually holds in probability. Moreover, the process $n^{\alpha / 2 - \eps} \int_{0}^{t} \left( K_{n} \left( s \right) - \pi \right) \, \dd s$ is continuous and has a continuous limit, so the convergence holds in the supremum metric, as required.
\end{proof}

In addition to a weak convergence result for the Dynkin martingale, the next theorem contains two important observations concerning the typical fluctuations of the state occupation measure around its limit. The first is that the fluctuations are of order $n^{- \alpha / 2}$ when the transition rates of the Markov chain are sped up with a factor $n^{\alpha}$. The second is that (after appropriate scaling) these fluctuations are well described by a Brownian motion whose predictable quadratic variation process strongly depends on the deviation matrix of the Markov chain.

\begin{theorem}
\label{thm:markovchainfclt}
Let $\alpha > 0$ and $d \in \mathbb{N}$. Let $\genQ$ be a $d \times d$ irreducible generator matrix and let $J_{n}$ be a continuous-time Markov chain with state space $\left\lbrace 1 , \dotsc , d \right\rbrace$, generator matrix $n^{\alpha} \genQ$, and state indicator function $K_{n}$. Let $Y_{n}$ denote the Dynkin martingale associated with $J_{n}$. Then, for $n \to \infty$, the stochastic process $n^{- \alpha/2} Y_{n}$ converges weakly to a Brownian motion $Y$ having predictable quadratic variation process
\begin{align}
\prequadvar{Y} \left( t \right)
&= - \left( \genQ^{\transpose} \diag{\pi} + \diag{\pi} \genQ \right) t. \label{eq:Dynkinlimitquadvar}
\end{align}
Additionally, for $n \to \infty$, the stochastic process
\begin{align}
n^{\alpha / 2} \int_{0}^{t} \left( K_{n} \left( s \right) - \pi \right) \, \dd s 
\end{align}
converges weakly to a Brownian motion $X$ having predictable quadratic variation process
\begin{align}
\prequadvar{X} \left( t \right) = \left( \diag{\pi} D + D^{\transpose} \diag{\pi} \right) t. \label{eq:Xlimitquadvar}
\end{align}
\end{theorem}
\begin{proof}
We know from the previous proof that the Dynkin martingale $Y_{n}$ satisfies
\begin{align*}
&\prequadvar{ n^{-\alpha / 2} Y_{n} } \left( t \right) \\
&\quad {} = n^{-\alpha} \prequadvar{Y_{n} } \left( t \right) \\
&\quad {} = \int_{0}^{t} \diag{\genQ^{\transpose} K_{n} \left( s \right)} \, \dd s - \int_{0}^{t} \genQ^{\transpose} \diag{K_{n} \left( s \right)} \, \dd s - \int_{0}^{t} \diag{K_{n} \left( s \right)} \genQ \, \dd s.
\end{align*}
It follows from Theorem \ref{thm:ergodic} that $\prequadvar{ n^{-\alpha / 2} Y_{n} }$ converges to
\begin{align*}
&\int_{0}^{t} \diag{\genQ^{\transpose} \pi} \, \dd s - \int_{0}^{t} \genQ^{\transpose} \diag{\pi} \, \dd s - \int_{0}^{t} \diag{\pi} \genQ \, \dd s \\
&\quad = - \int_{0}^{t} \genQ^{\transpose} \diag{\pi} \, \dd s - \int_{0}^{t} \diag{\pi} \genQ \, \dd s \\
&\quad = - \left( \genQ^{\transpose} \diag{\pi} + \diag{\pi} \genQ \right) t
\end{align*}
uoc in probability. The penultimate equality is based on the fact that $\pi^{\transpose} \genQ = 0$.
Using the same arguments as in the previous proof, we conclude that $n^{-\alpha / 2} Y_{n}$ converges weakly to a Brownian motion $Y$ and that its compensator $\prequadvar{Y}$ is given by Eq.\ \eqref{eq:Dynkinlimitquadvar}.

Then the process $- n^{\alpha/2} \int_{0}^{t} \genQ^{\transpose} K_{n} \left( s \right) \, \dd s$ must converge weakly to $Y$ as well. It follows that the process
\begin{align*}
n^{\alpha / 2} \int_{0}^{t} \left( K_{n} \left( s \right) - \pi \right) \, \dd s = - n^{\alpha/2} \int_{0}^{t} F^{\transpose} \genQ^{\transpose} K_{n} \left( s \right) \, \dd s
\end{align*}
converges weakly to a Brownian motion $X$ with
\begin{align*}
\prequadvar{X} \left( t \right)
= F^{\transpose} \left( - \left( \genQ^{\transpose} \diag{\pi} + \diag{\pi} \genQ \right) t \right) F 
= \left( \diag{\pi} D + D^{\transpose} \diag{\pi} \right) t.
\end{align*}
For a justification of the last equality, see Eq.\ \eqref{eq:quadvardevmat}.
\end{proof}

\section{Main result}
\label{sec:markovintegral}
We have settled weak convergence of the Dynkin martingale and the state occupation measure of a Markov chain in Theorem \ref{thm:markovchainfclt}. Motivated by the analysis of modulated queueing systems, we are also interested in the convergence of stochastic integrals with respect to the Dynkin martingale and the state occupation measure of a Markov chain.

As mentioned before, the convergence of stochastic integrals with respect to semimartingales is a very delicate subject. For concreteness, suppose that $X_{n}$ is some semimartingale and $H_{n}$ is a suitable integrand. Then, even when $H_{n}$ and $X_{n}$ are well-behaved deterministic processes converging uniformly to the zero process, the stochastic integral $H_{n} \stdot X_{n}$ may not converge as $n \to \infty$. 

Nevertheless, there are two well-known cases in which the analysis simplifies considerably. The first case is when $X_{n}$ is a martingale. The second case (partly covering the first) is when $X_{n}$ satisfies the so-called P-UT condition. The term P-UT stands for `Predictably Uniformly Tight'; see \cite[Def.\ VI.6.1]{js2003} and \cite{KP1991} for definitions and some explanation. 

When $X_{n}$ is the Dynkin martingale, we find ourselves in a situation that is covered by both the first case and the second case. Then we may use fairly standard arguments to establish convergence of the stochastic integral. 

However, when integrating against the state occupation measure, neither the first nor the second case applies. We will get around this problem by restricting the integrands to be processes of finite variation which converge in a specific way. Under this restriction, we can exploit properties of both the Dynkin martingale and the state indicator function to obtain weak convergence of the stochastic integral with respect to the state occupation measure.

We start with some assumptions and notation, following mainly \cite[p.\ 204]{js2003}. Let $X$ be a $d$-dimensional locally square-integrable martingale with respect to a filtration $\mathds{F}$ (as defined in \cite{js2003}). For simplicity, we assume that
\begin{align*}
\prequadvar{X} \left( t \right)
&= \int_{0}^{t} C \left( s \right) \, \dd s
\end{align*}
for a predictable process $C$ taking values in the set of all symmetric nonnegative $d \times d$ matrices. We denote by $L^{2}_{\mathrm{loc}} \left( X \right)$ the set of predictable processes $H$ taking values in $\mathbb{R}^{k \times d}$ such that the process
\begin{align*}
\int_{0}^{t} H \left( s \right) C \left( s \right) H \left( s \right)^{\transpose} \, \dd s
\end{align*}
is locally integrable.

Under this set of assumptions, \cite[Th.\ VI.6.4]{js2003} guarantees the existence of the stochastic integral $H \stdot X$ for $H \in L^{2}_{\mathrm{loc}} \left( X \right)$ and shows that $H \stdot X$ is a locally square-integrable martingale with
\begin{align*}
\prequadvar{H \stdot X} \left( t \right)
&= \int_{0}^{t} H \left( s \right) C \left( s \right) H \left( s \right)^{\transpose} \, \dd s.
\end{align*}
The following theorem describes the asymptotic behavior of this stochastic integral when $X$ is given by the Dynkin martingale of a Markov chain. %To be able to apply the MCLT in its proof, the theorem puts some restrictions on the integrand.

\begin{theorem}
\label{thm:stochintconvmart}
Let $\alpha > 0$ and $d \in \mathbb{N}$. Let $\genQ$ be a $d \times d$ irreducible generator matrix and let $J_{n}$ be a continuous-time Markov chain with state space $\left\lbrace 1 , \dotsc , d \right\rbrace$, generator matrix $n^{\alpha} \genQ$, and state indicator function $K_{n}$. Let $Y_{n}$ denote the Dynkin martingale associated with $J_{n}$. Let $H_{n}$ be a c\`{a}dl\`{a}g adapted process such that $H_{n}^{-} \in L^{2}_{\mathrm{loc}} \left( Y_{n} \right)$, where $H_{n}^{-} \left( t \right) = H_{n} \left( t- \right)$. Assume that $H_{n}$ converges to $H$ uoc in probability, where $H$ is a deterministic continuous function. 
%Additionally, assume that $n^{- \alpha} \expectation{ \sup_{0 \leq s \leq t} \norm{ H_{n} \left( s \right) }^{2} }$ converges to $0$. 
Then, for $n \to \infty$, the stochastic integral $H_{n}^{-} \stdot n^{- \alpha / 2} Y_{n}$ converges weakly to the stochastic integral $H \stdot Y$, with $Y$ being a Brownian motion whose predictable quadratic variation process is given by Eq.\ \eqref{eq:Dynkinlimitquadvar}.
\end{theorem}
\begin{proof}
Recall that $L^{2}_{\mathrm{loc}} \left( Y_{n} \right)$ is a collection of predictable processes. This is why $H_{n}^{-}$ is used as an integrand rather than $H_{n}$: the process $H_{n}^{-}$ is predictable, whereas $H_{n}$ may not be predictable. Since $H$ is deterministic and continuous, it is obviously predictable, so there is no need to use $H^{-}$ in the limiting stochastic integral.

We know from Theorem \ref{thm:markovchainfclt} that $n^{- \alpha / 2} Y_{n}$ converges weakly to a Brownian motion $Y$ with $\prequadvar{Y}$ satisfying Eq.\ \eqref{eq:Dynkinlimitquadvar}. Because $H$ is a deterministic continuous function, we obtain weak convergence of $\left( H_{n}^{-} , n^{- \alpha / 2} Y_{n} \right)$ to $\left( H , Y \right)$.

We would like to apply \cite[Th.\ VI.6.22]{js2003} to show weak convergence of $ H_{n}^{-} \stdot n^{- \alpha / 2} Y_{n}$ to $H \stdot Y$. To be able to apply this result, we need to verify that the sequence of martingales $n^{- \alpha / 2} Y_{n}$ has the P-UT property. The validity of this property follows from \cite[Cor.\ VI.6.29]{js2003}, because $n^{- \alpha / 2} Y_{n}$ is a martingale converging weakly to $Y$ and its jumps are bounded by $1$. Thus, \cite[Th.\ VI.6.22]{js2003} gives us the weak convergence of $ H_{n}^{-} \stdot n^{- \alpha / 2} Y_{n}$ to $H \stdot Y$.
\end{proof}

Now we have conditions under which the stochastic integral with respect to the Dynkin martingale converges weakly. This is exploited in the proof of the next theorem, which states that, under the proviso that the integrand converges nicely, certain stochastic integrals with respect to the state occupation measure converge weakly. The proof of this result relies on showing that the stochastic integral with respect to the state occupation measure is asymptotically equivalent to the same stochastic integral with respect to the Dynkin martingale. As we already have established weak convergence of the stochastic integral with respect to the Dynkin martingale in the previous theorem, we immediately get weak convergence of the stochastic integrals with respect to the state occupation measure.

\begin{theorem}
\label{thm:stochintstateind}
Impose the conditions of Theorem \ref{thm:stochintconvmart}, together with the extra requirement that each entry of $n^{- \alpha / 2} H_{n}$ is a finite variation process whose total variation process converges to the zero process uoc in probability. Then the stochastic process
\begin{align*}
\int_{0}^{t} H_{n} \left( s \right) n^{\alpha / 2} \genQ^{\transpose} K_{n} \left( s \right) \, \dd s
\end{align*}
converges weakly to the stochastic integral $H \stdot Y$.
\end{theorem}
\begin{proof}
First, recall the form of $n^{- \alpha /2} Y_{n}$ (cf.\ Eq.\ \eqref{eq:dynkin}) and observe that
\begin{align*}
&\int_{0}^{t} H_{n} \left( s \right) n^{\alpha / 2} \genQ^{\transpose} K_{n} \left( s \right) \, \dd s \\
&\quad {} = \int_{0}^{t} H_{n}^{-} \left( s \right) n^{\alpha / 2} \genQ^{\transpose} K_{n} \left( s \right) \, \dd s \\
&\quad {} = \int_{0}^{t} H_{n}^{-} \left( s \right) \, \dd n^{- \alpha / 2} K_{n} \left( s \right) - \int_{0}^{t} H_{n}^{-} \left( s \right) \, \dd n^{- \alpha / 2} Y_{n} \left( s \right).
\end{align*}
Theorem \ref{thm:stochintconvmart} asserts that $H_{n}^{-} \stdot n^{- \alpha / 2} Y_{n}$ converges weakly to $H \stdot Y$, so it suffices to prove that $H_{n}^{-} \stdot n^{- \alpha / 2} K_{n}$ converges to the zero process uoc in probability. To this end, it suffices to prove that 
\begin{align*}
\int_{0}^{t} H_{n}^{-} \left( i,j ; s \right) \, \dd n^{- \alpha / 2} K_{n} \left( j ; s \right)
&= \int_{0}^{t} n^{- \alpha / 2} H_{n}^{-} \left( i,j ; s \right) \, \dd \ind{J_{n} \left( s \right) = j}
\end{align*}
converges to the zero process uoc in probability.

Denote the total variation process of $n^{- \alpha / 2} H_{n}^{-} \left( i,j ; t \right)$ by $V_{n} \left( i,j ; t \right)$; it is clearly bounded by the total variation process of $n^{- \alpha / 2} H_{n} \left( i,j ; t \right)$. 
Now the crucial observation is that the process $\ind{J_{n} \left( t \right) = j}$ consists of alternating jumps $+1$ and $-1$, which implies that
\begin{align*}
&\sup_{0 \leq s \leq t} \abs{ \int_{0}^{s} n^{- \alpha / 2} H_{n}^{-} \left( i,j ; r \right) \, \dd \ind{J_{n} \left( r \right) = j} } \\
&\quad {} \leq V_{n} \left( i,j ; t \right) + \sup_{0 \leq s \leq t} \abs{n^{- \alpha / 2} H_{n}^{-} \left( i,j ; s \right)}.
\end{align*}
The validity of this inequality is a consequence of Lemma \ref{lem:alternatingjumps}.
Because both $n^{- \alpha / 2} H_{n}^{-}$ and $V_{n}$ converge to the zero process uoc in probability, also $H_{n}^{-} \stdot n^{- \alpha / 2} K_{n}$ converges to the zero process uoc in probability, as required.
\end{proof}

The next theorem concerns weak convergence of a vector of stochastic integrals. For each of these stochastic integrals, the integrator is the scaled and centered state occupation measure from Theorem \ref{thm:markovchainfclt}. The proof of weak convergence in this case follows the proofs of Theorem \ref{thm:stochintconvmart} and Theorem \ref{thm:stochintstateind} quite closely.

\begin{theorem}
\label{thm:vectorstateoccupationintegral}
Impose the conditions of Theorem \ref{thm:stochintconvmart} and define
\begin{align*}
G_{n} \left( t \right) &= n^{\alpha/2} \int_{0}^{t} \left( K_{n} \left( s \right) - \pi \right) \, \dd s.
\end{align*}
For some fixed $m \in \mathbb{N}$, let $H_{1,n} , \dotsc , H_{m,n}$ be c\`{a}dl\`{a}g, adapted processes such that $H_{k,n}^{-} \in L^{2}_{\mathrm{loc}} \left( Y_{n} \right)$, where $H_{k,n}^{-} \left( t \right) = H_{k,n} \left( t- \right)$. Assume that each $H_{k,n}^{-}$ converges to $H_{k}$ uoc in probability, where $H_{k}$ is a deterministic, continuous function. Additionally, assume that each entry of $n^{-\alpha/2} H_{k,n}$ is a finite variation process whose total variation process converges to the zero process uoc in probability. Then, for $n \to \infty$, the vector of stochastic integrals
\begin{align}
\left( H_{1,n}^{-} \stdot G_{n} , \dotsc , H_{m,n}^{-} \stdot G_{n} \right)
\label{eq:vectorstochintprelimit}
\end{align}
converges weakly to the vector of stochastic integrals
\begin{align}
\left( H_{1} \stdot X , \dotsc , H_{m} \stdot X \right),
\label{eq:vectorstochintlimit}
\end{align}
with $X$ being a Brownian motion whose predictable quadratic variation process is given by Eq.\ \eqref{eq:Xlimitquadvar}.
\end{theorem}
\begin{proof}
We know from Eq.\ \eqref{eq:FtransposeQtransposeKn} that
\begin{align*}
G_{n} \left( t \right) &= n^{\alpha/2} \int_{0}^{t} \left( K_{n} \left( s \right) - \pi \right) \, \dd s = - \int_{0}^{t} F^{\transpose} n^{\alpha/2} \genQ^{\transpose} K_{n} \left( s \right) \, \dd s,
\end{align*}
so
\begin{align*}
\left( H_{k,n}^{-} \stdot G_{n} \right) \left( t \right) 
&= \int_{0}^{t} H_{k,n}^{-} \left( s \right) n^{\alpha/2} \left( K_{n} \left( s \right) - \pi \right) \, \dd s \\
&= - \int_{0}^{t} H_{k,n}^{-} \left( s \right) F^{\transpose} n^{\alpha/2} \genQ^{\transpose} K_{n} \left( s \right) \, \dd s.
\end{align*}
It follows from the proof of Theorem \ref{thm:stochintstateind} that the last integral equals
\begin{align*}
\int_{0}^{t} H_{k,n}^{-} \left( s \right) F^{\transpose} \, \dd n^{-\alpha/2} Y_{n} \left( s \right) + R_{k,n} \left( t \right),
\end{align*}
where
\begin{align*}
R_{k,n} \left( t \right) 
&= - \int_{0}^{t} H_{k,n}^{-} \left( s \right) F^{\transpose} \, \dd n^{-\alpha/2} K_{n} \left( s \right).
\end{align*}
Clearly, $H_{k,n}^{-} F^{\transpose}$ is a process in $L^{2}_{\mathrm{loc}} \left( Y_{n} \right)$ and converges to $H_{k} F^{\transpose}$ uoc in probability. Moreover, each entry of $n^{-\alpha/2} H_{k,n} F^{\transpose}$ is a finite variation process whose total variation process converges to the zero process uoc in probability. Consequently, the proof of Theorem \ref{thm:stochintstateind} also gives us convergence of $R_{k,n}$ to the zero process uoc in probability.

Now observe that the vector of stochastic integrals in Eq.\ \eqref{eq:vectorstochintprelimit} equals
\begin{align*}
\left( H_{1,n}^{-} F^{\transpose} \stdot n^{-\alpha/2} Y_{n} + R_{1,n}, \dotsc , H_{m,n}^{-} F^{\transpose} \stdot n^{-\alpha/2} Y_{n} + R_{m,n} \right).
\end{align*}
Due to the processes $R_{k,n}$ converging to the zero process, it suffices to show that
\begin{align}
\left( H_{1,n}^{-} F^{\transpose} \stdot n^{-\alpha/2} Y_{n}, \dotsc , H_{m,n}^{-} F^{\transpose} \stdot n^{-\alpha/2} Y_{n} \right)
\label{eq:vectorstochintsufficient}
\end{align}
converges weakly to the limiting vector of stochastic integrals in Eq.\ \eqref{eq:vectorstochintlimit}.

To show weak convergence of the vector in Eq.\ \eqref{eq:vectorstochintsufficient}, we may follow the proof of Theorem \ref{thm:stochintconvmart}. Recall that $Y_{n}$ denotes the Dynkin martingale and that  $n^{-\alpha/2} Y_{n}$ converges weakly to a Brownian motion $Y$, whose predictable quadratic variation process is given by Eq.\ \eqref{eq:Dynkinlimitquadvar}. For notational convenience, we define $Y_{n}^{\left( k \right)} = Y_{n}$ and $Y^{\left( k \right)} = Y$ for $k = 1 , \dotsc , m$. Because the processes $H_{k,n}^{-} F^{\transpose}$ converge uoc in probability to the deterministic functions $H_{k} F^{\transpose}$, we get weak convergence of
\begin{align*}
\left( H_{1,n}^{-} F^{\transpose} , \dotsc , H_{m,n}^{-} F^{\transpose} , n^{-\alpha/2} Y_{n}^{\left( 1 \right)} , \dotsc , n^{-\alpha/2} Y_{n}^{\left( m \right)} \right)
\end{align*}
to
\begin{align*}
\left( H_{1}^{-} F^{\transpose} , \dotsc , H_{m}^{-} F^{\transpose} , Y^{\left( 1 \right)} , \dotsc , Y^{\left( m \right)} \right).
\end{align*}
As we have shown in the proof of Theorem \ref{thm:stochintconvmart}, the sequence of martingales $n^{-\alpha/2} Y_{n}$ has the P-UT property. Then \cite[Th.\ VI.6.22]{js2003} implies the weak convergence of the vector of stochastic integrals in Eq.\ \eqref{eq:vectorstochintsufficient} to
\begin{align*}
\left( H_{1} F^{\transpose} \stdot Y, \dotsc , H_{m} F^{\transpose} \stdot Y \right) = \left( H_{1} \stdot X , \dotsc , H_{m} \stdot X \right),
\end{align*}
where $X = F^{\transpose} Y$. Note that $X$ is a Brownian motion and that its predictable quadratic variation process is indeed given by Eq.\ \eqref{eq:Xlimitquadvar}, as required.
\end{proof}

%\begin{corollary}
%\label{cor:zerostateoccupationintegral}
%Impose the conditions of the previous theorem, together with the extra requirement that $H_{n}$ converges to the zero process uoc in probability. Then the stochastic process
%\begin{align*}
%\int_{0}^{t} H_{n} \left( s \right) n^{\alpha/2} \left( K_{n} \left( s \right) - \pi \right) \, \dd s
%\end{align*}
%converges to the zero process uoc in probability.
%\end{corollary}
%\begin{proof}
%We know from Eq.\ \eqref{eq:FtransposeQtransposeKn} that
%\begin{align*}
%\int_{0}^{t} \left( K_{n} \left( s \right) - \pi \right) \, \dd s = - \int_{0}^{t} F^{\transpose} \genQ^{\transpose} K_{n} \left( s \right) \, \dd s,
%\end{align*}
%so
%\begin{align*}
%\int_{0}^{t} H_{n} \left( s \right) n^{\alpha/2} \left( K_{n} \left( s \right) - \pi \right) \, \dd s 
%&= - \int_{0}^{t} H_{n} \left( s \right) F^{\transpose} n^{\alpha/2} \genQ^{\transpose} K_{n} \left( s \right) \, \dd s.
%\end{align*}
%It is easy to see that the process $H_{n} F^{\transpose}$ converges to the zero process uoc in probability and that it satisfies the requirements of Theorem \ref{thm:stochintstateind}. An application of this theorem immediately gives the desired result.
%\end{proof}

\section{Functions of bounded variation}
\label{sec:boundedvariation}
We say that a function $x \colon \left[ 0,\infty \right) \to \mathbb{R}$ is of bounded variation if
\begin{align*}
v_{x} \left( t \right)
&= \sup \sum_{k=1}^{n} \abs{ x \left( t_{k} \right) - x \left( t_{k-1} \right) } < \infty
\end{align*}
for each fixed $t > 0$, where the supremum runs over all partitions $0 \leq t_{0} < t_{1} < \dotsc < t_{n} \leq t$. If $x$ is of bounded variation, then $v_{x}$ is called its total variation.
In the case that $x$ is of bounded variation, there exists a unique decomposition
\begin{align*}
x = x^{+} - x^{-},
\end{align*}
where $x^{+}$ and $x^{-}$ are nondecreasing and satisfy $v_{x} = x^{+} + x^{-}$ (cf.\ \cite[Pr.\ I.3.3]{js2003}). The following lemma is obvious.
\begin{lemma}
\label{lem:bddvarisbdd}
A function of bounded variation is bounded on compact intervals and so is its total variation.
\end{lemma}
An interesting question is what happens when functions of bounded variation are mapped to other functions. The next two results describe situations in which the property of bounded variation is retained.
\begin{lemma}
\label{lem:bddvarctufunc}
Let $x \colon \left[ 0,\infty \right) \to \mathbb{R}$ be of bounded variation and let $f \colon \mathbb{R} \to \mathbb{R}$ be a continuous function that is Lipschitz continuous on compact intervals. Then $y \left( t \right) = f \left( x \left( t \right) \right)$ is of bounded variation. Additionally, given $T > 0$, there exists $c_{T} > 0$ such that $v_{y} \left( t \right) \leq c_{T} v_{x} \left( t \right)$ for all $t \in \left[ 0,T \right]$.
\end{lemma}
\begin{proof}
Fix $T > 0$. Lemma \ref{lem:bddvarisbdd} implies that $x$ is bounded by a constant $B > 0$ on $\left[ 0,T \right]$. The function $f$ is Lipschitz continuous on $\left[ -B , B \right]$ with Lipschitz constant $c > 0$. It follows that $\abs{ y \left( t \right) - y \left( s \right) } \leq c \abs{ x \left( t \right) - x \left( s \right) }$ for all $s,t \in \left[ 0,t \right]$. Hence, the total variation of $y$ over $\left[ 0,T \right]$ is bounded by $c$ times the total variation of $x$ over $\left[ 0,T \right]$.
\end{proof}

\begin{lemma}
\label{lem:bddvarmultiplication}
Let $x \colon \left[ 0,\infty \right) \to \mathbb{R}$ and $y \colon \left[ 0,\infty \right) \to \mathbb{R}$ be of bounded variation. Then $z \left( t \right) = x \left( t \right) y \left( t \right)$ is of bounded variation. Additionally, given $T > 0$, there exists $c_{T} > 0$ such that $v_{z} \left( t \right) \leq c_{T} \left( v_{x} \left( t \right) + v_{y} \left( t \right) \right)$ for all $t \in \left[ 0,T \right]$.
\end{lemma}
\begin{proof}
Fix $T > 0$. By Lemma \ref{lem:bddvarisbdd}, both $x$ and $y$ are bounded by a constant $B > 0$ on $\left[ 0,T \right]$. Then
\begin{align*}
\abs{ z \left( t \right) - z \left( s \right) }
&= \abs{ x \left( t \right) y \left( t \right) - x \left( s \right) y \left( s \right) } \\
&= \abs{ x \left( t \right) y \left( t \right) - x \left( t \right) y \left( s \right) + x \left( t \right) y \left( s \right) - x \left( s \right) y \left( s \right) } \\
&\leq \abs{ x \left( t \right) } \abs{ y \left( t \right) - y \left( s \right) } + \abs{ x \left( t \right) - x \left( s \right) } \abs{ y \left( s \right) } \\
&\leq B \abs{ y \left( t \right) - y \left( s \right) } + B \abs{ x \left( t \right) - x \left( s \right) }
\end{align*}
 for all $s,t \in \left[ 0,t \right]$. Hence, the total variation of $z$ over $\left[ 0,T \right]$ is bounded by $B$ times the sum of the total variation of $x$ and the total variation of $y$ over $\left[ 0,T \right]$.
\end{proof}

\begin{lemma}
\label{lem:alternatingjumps}
Let $y \colon \left[ 0,\infty \right) \to \mathbb{R}$ be a function of bounded variation and let $x \in D \left( \left[ 0,\infty \right) ; \left\lbrace 0,1 \right\rbrace \right)$. Then the Lebesgue-Stieltjes integral $y \stdot x$ satisfies
\begin{align}
\sup_{0 \leq t \leq T} \abs{ \int_{0}^{t} y \left( s \right) \, \dd x \left( s \right) } 
&\leq v_{y} \left( T \right) + \sup_{0 \leq t \leq T}  \abs{ y \left( t \right) } \label{eq:altjumps}
\end{align}
for every fixed $T > 0$.
\end{lemma}
\begin{proof}
We only have to prove that
\begin{align}
\abs{ \int_{0}^{t} y \left( s \right) \, \dd x \left( s \right) } 
&\leq v_{y} \left( t\right) + \sup_{0 \leq s \leq t}  \abs{ y \left( s \right) }, \label{eq:altjumpsnonsup}
\end{align}
for fixed $t \geq 0$, because $t \mapsto v_{y} \left( t\right)$ and $t \mapsto \sup_{0 \leq s \leq t}  \abs{ y \left( s \right) }$ are nondecreasing.

Clearly, $x$ has alternating jumps of size $+1$ and $-1$, and is constant between jumps. Thus, if $x$ has at most one jump in $\left[ 0,t \right]$, then Eq.\ \eqref{eq:altjumpsnonsup} is trivial.

Suppose that $x$ has exactly $2m$ jumps in $\left[ 0,t \right]$ (where $m \in \mathbb{N}$) and denote the corresponding jump times by $0 < s_{1} < \dotsc < s_{2m} \leq t$. If the first jump equals $+1$, then
\begin{align*}
\int_{0}^{t} y \left( s \right) \, \dd x \left( s \right) 
&= \sum_{k=0}^{m-1} \left( y \left( s_{2k + 1} \right) - y \left( s_{2k + 2} \right) \right),
\end{align*}
so
\begin{align}
\abs{ \int_{0}^{t} y \left( s \right) \, \dd x \left( s \right) } 
&\leq \sum_{k=0}^{m-1} \abs{ y \left( s_{2k + 2} \right) - y \left( s_{2k + 1} \right) } \leq v_{y} \left( t \right). \label{eq:altjumpseven}
\end{align}
If the first jump equals $-1$, then
\begin{align*}
\int_{0}^{t} y \left( s \right) \, \dd x \left( s \right) 
&= \sum_{k=0}^{m-1} \left( -y \left( s_{2k + 1} \right) + y \left( s_{2k + 2} \right) \right),
\end{align*}
so Eq.\ \eqref{eq:altjumpseven} holds in this case, too. Hence, Eq.\ \eqref{eq:altjumpsnonsup} holds when $x$ has an even number of jumps in $\left[ 0,t \right]$.

Suppose that $x$ has exactly $2m+1$ jumps in $\left[ 0,t \right]$ (where $m \in \mathbb{N}$) and denote the corresponding jump times by $0 < s_{1} < \dotsc < s_{2m} < s_{2m+1} \leq t$. Taking $\delta = \left( s_{2m+1} - s_{2m} \right)/2$, we get
\begin{align*}
\abs{ \int_{0}^{t} y \left( s \right) \, \dd x \left( s \right) }
&= \abs{ \int_{0}^{s_{2m} + \delta} y \left( s \right) \, \dd x \left( s \right) + \int_{s_{2m} + \delta}^{t} y \left( s \right) \, \dd x \left( s \right) } \\
&\leq \abs{ \int_{0}^{s_{2m} + \delta} y \left( s \right) \, \dd x \left( s \right) } + \abs{ y \left( s_{2m+1} \right) } \\
&\leq v_{y} \left( s_{2m} \right) + \abs{ y \left( s_{2m+1} \right) } \\
&\leq v_{y} \left( t \right) + \sup_{0 \leq s \leq t} \abs{ y \left( s \right) }.
\end{align*}
Hence, Eq.\ \eqref{eq:altjumpsnonsup} also holds when $x$ has an odd number of jumps in $\left[ 0,t \right]$.
\end{proof}

%\clearpage
\bibliographystyle{plain}
\bibliography{ref}

\end{document}